\documentclass[11pt,english]{article}
\usepackage[margin = 2.5cm]{geometry}

\usepackage{amsmath}
\usepackage{amsthm}
\usepackage{amssymb}
\usepackage{mathtools}
\usepackage{graphicx}
\usepackage{xcolor}
\usepackage[hidelinks, bookmarks = false]{hyperref}
\usepackage{enumitem}
\usepackage{bbm}

\usepackage{caption}
\newtheorem{theorem}{Theorem}[section]

\newtheorem{proposition}[theorem]{Proposition}

\newtheorem{lemma}[theorem]{Lemma}

\newtheorem{definition}[theorem]{Definition}

\newtheorem{conjecture}[theorem]{Conjecture}

\def\b1K{\mbox{\boldmath $K$}_{-1}}
\def\bK{\mbox{\boldmath $K$}}

\newbox\noforkbox \newdimen\forklinewidth
\forklinewidth=0.3pt \setbox0\hbox{$\textstyle\smile$}
\setbox1\hbox to \wd0{\hfil\vrule width \forklinewidth depth-2pt
 height 10pt \hfil}
\wd1=0 cm \setbox\noforkbox\hbox{\lower 2pt\box1\lower
2pt\box0\relax}

\def\sub'm{\prec_{\bK'}}
\def\grpf #1 #2{{\rm grp}_{#2}(#1)}
\def\fldf #1 #2{{\rm fld}_{#2}(#1)}
\def\dclf #1 #2{{\rm dcl}_{#2}(#1)}
\def\rclf #1 #2{{\rm rcl}_{#2}(#1)}
\def\aclf #1 #2{{\rm acl}_{#2}(#1)}
\def\acff #1 #2{{\rm acf}_{#2}(#1)}
\def\strf #1 #2{{\rm str}_{#2}(#1)}
\def\tclf #1 #2{{\rm acf}_{#2}(#1)}

\def\hbar{{\bf h}}

\date{\today}

\newcommand{\F}{\mathcal{F}}

\newcommand{\G}{\mathcal{G}}
\newcommand{\T}{\mathcal{H}}
\newcommand{\D}{\mathcal{D}}

\newcommand{\Z}{\mathcal{Z}}

\newcommand{\Sc}{\mathcal{S}}
\newcommand{\cH}{\mathcal{H}}

\newcommand{\Pc}{\mathcal{P}}

\newcommand{\al}{\alpha}

\newcommand{\ve}{\varepsilon}

\newcommand{\ex}{\mathrm{ex}}
\newcommand{\bip}{\mathrm{bip}}

\usepackage{mathtools}

\newtheorem{prob}{Problem}

\usepackage{xpatch}
\xpatchcmd{\proof}{\itshape}{\normalfont\proofnamefont}{}{}
\newcommand{\proofnamefont}{}
\renewcommand{\proofnamefont}{\bfseries}

\title{Induced rational exponents near two}

\author{Tao Jiang \footnote{Dept. of Mathematics, Miami University, Oxford, OH 45056, USA, {\tt jiangt@miamioh.edu}.  }\and Sean Longbrake \footnote{Dept. of Mathematics, Emory University,  Atlanta, GA 30322, USA {\tt sean.longbrake@emory.edu}} }
\begin{document}

\maketitle

\begin{abstract}
Given a bipartite graph $H$ and a natural number $s$, 
let $\ex^*(n,H,s)$ denote the maximum number of edges in an $n$-vertex graph that contains neither
$K_{s,s}$ nor an induced copy of $H$. Hunter, Milojevi\'c, Sudakov, and Tomon \cite{HMST}  conjectured that
$\ex^*(n,H,s)=O_{H,s}(\ex(n,H))$ whenever $H$ is connected. Motivated by this conjecture and the rational exponents conjecture, Dong, Gao, Li, and Liu \cite{DGLL} conjectured that for every rational $r\in (1,2)$ there is a bipartite graph $H$ and an $s_0$ such that $\ex^*(n,H,s)=\Theta(n^r)$ for all $s\geq s_0$.

We prove that the latter conjecture holds for all rationals $r=2-a/b$, where $a,b\in\mathbb{N}$ satisfy $b\geq \max\{a,(a-1)^2\}$. Our result extends a well-known result of Conlon and Janzer \cite{CJ} to the induced setting and adds more evidence to support the former conjecture.
\end{abstract}

\section{Introduction}

Given a family $\cH$ of graphs, a graph $G$ is {\it $\cH$-free} if $G$ does not contain any member of $\cH$ as a subgraph. The {\it extremal number} $\ex(n,\cH)$ of $\cH$ (also called the {\it Tur\'an number} of $\cH$) is the largest number of edges in an $\cH$-free $n$-vertex graph.
When $\cH$ consists of a single graph $H$, we write $\ex(n,H)$ for $\ex(n,\cH)$. The study of the function $\ex(n,\cH)$ is usually referred
to as the {\it Tur\'an problem}, and it is one of the central problems in extremal graph theory. When $\cH$ consisting only of non-bipartite graphs, the celebrated
Erd\H{o}s-Stone-Simonovits theorem \cite{ES, E-Stone} determines the asymptotics of $\ex(n,\cH)$. When $\cH$ contains a bipartite graph,
much less is known in general. Nevertheless, in recent years, much progress has been made on the Tur\'an problem for bipartite graphs,
particularly on the so-called {\it rational exponents conjecture} of Erd\H{o}s and Simonovits (see, for example \cite{Erdos-problem}).

\begin{conjecture}[(Rational exponents conjecture] \label{conj:ES}
For every rational number $r\in [1,2]$, there exists a graph $H$ with $\ex(n,H)=\Theta(n^r)$.
\end{conjecture}

The biggest breakthrough on the conjecture was made by Bukh and Conlon \cite{BC} who showed that for any rational number $r\in [1,2]$
there exists a finite family $\cH$ of graphs such that $\ex(n,\cH)=\Theta(n^r)$. Following this breakthrough, much progress has been made on
the conjecture in its original form (which asks for a single graph $H$ rather than a family $\cH$) by various authors (see \cite{CJL, Janzer-bipartite, JJM, JMY, Jiang-Qiu2, KKL} for instance) and Conjecture \ref{conj:ES} has now been verified for many values of $r$. In particular, Conlon and Janzer \cite{CJ} 
showed the following.
\begin{theorem} [\cite{CJ}] \label{thm:CJ}
For each rational number of the form $r=2-a/b$, where $a,b$ are natural numbers with $b\geq \{a, (a-1)^2\}$,
there exists a bipartite graph $H$ with $\ex(n,H)=\Theta(n^r)$.
\end{theorem}

Very recently, motivated by applications in discrete geometry and structural graph theory, Hunter, Milojevi\'c, Sudakov, and Tomon \cite{HMST} initiated
a systematic study of the following induced variant of the Tur\'an problem. For a positive integer $s$, let $K_{s,s}$ denote the complete bipartite graph
with $s$ vertices in each part. Given positive integers $n,s$ and a family $\cH$ of graphs, the {\it induced Tur\'an number}
of $\cH$, denoted by $\ex^*(n,\cH,s)$, is the maximum number of edges in an $n$-vertex graph that contains neither a copy of $K_{s,s}$ nor  an induced copy of any member of $\cH$. When $\cH$ consists of a single graph $H$, we write $\ex^*(n,H,s)$ for $\ex^*(n,\cH,s)$.
See \cite{HMST} for a detailed account of the connection of the study of this function to various problems in discrete geometry and structural graph theory (in particular, to the Erd\H{o}s-Hajnal conjecture). When $s$ is sufficiently large, the definition immediately implies that $\ex^*(n,H,s)\geq \ex(n,H)$. Hunter, Milojevi\'c,
Sudakov, and Tomon \cite{HMST} conjectured the following.
\begin{conjecture}[\cite{HMST}] \label{conj:HMST}
For any connected bipartite graph $H$, in fact we must have $\ex^*(n,H,s)\leq C_H(s)\cdot \ex(n,H)$ for
some $C_H(s)$ depending only on $H$ and $s$. 
\end{conjecture}
They provided evidence for the conjecture by considering trees, cycles, the cube graph, and bipartite graphs with degree bounded by $k$ on one side. One of the main results obtained in \cite{HMST} is
\begin{theorem} [\cite{HMST}]
Let $H$ be a bipartite graph with parts $A,B$ such that each vertex in $B$ has degree at most $k$. Then $\ex^*(n,H,s)\leq O_{H,s}(n^{2-1/k})$.
\end{theorem}
In particular, this extends a well-known result of F\"uredi \cite{Furedi} and Alon, Krivelevich, and Sudakov \cite{AKS} which states that $\ex(n,H)=O_H(n^{2-1/k})$ for such an $H$.

Motivated  by the theorem of Bukh and Conlon \cite{BC} and the induced Tur\'an problem, Dong, Gao, Li, and Liu \cite {DGLL} 
established the following extension of the Bukh-Conlon Theorem to the induced setting.
\begin{theorem} [\cite{DGLL}]
For every rational $r=\frac{a}{b}\in (1,2)$, where $a,b$ are positive integers, there exists a family $\cH$ consisting of at most $2^a$ bipartite graphs 
such that $\ex^*(n,\cH,s)=\Theta(n^r)$.
\end{theorem}
They also put forth the following induced version of the rational exponents conjecture.

\begin{conjecture} [Induced Rational Exponents Conjecture, \cite{DGLL} Conjecture 1.1] \label{conj:DGLL}
For every rational number $r\in (1,2)$, there exist a bipartite graph $H$ and
a constant $s_0$ such that $\ex^*(n,H,s)=\Theta_s(n^r)$ for any $s\geq s_0$.
\end{conjecture}
We remark that though the $s\geq s_0$ condition was not explicitly stated in \cite{DGLL}, it was implicit in the discussions. 
Dong, Gao, Li, and Liu  \cite{DGLL} also provided further evidence to Conjecture \ref{conj:HMST} by proving optimal bounds 
for the maximum size of $K_{s,s}$-free graphs without an induced copy of theta graphs or prism graphs.

In this paper, we establish the induced extension of the theorem of Conlon and Janzer (Theorem \ref{thm:CJ}), and thereby
prove Conjecture \ref{conj:DGLL} for all rationals $r\in (1,2)$ of the form $2-a/b$, where $a,b$ are positive integers where $b\geq \max\{a, (a-1)^2\}$.

\begin{theorem} [Main theorem] \label{thm:main}
For each rational number of the form $r=2-a/b$ where $a,b$ are natural numbers with $b\geq \{a, (a-1)^2\}$,
there exist a bipartite graph $H$ and a constant $s_0$ such that $\ex^*(n,H,s)=\Theta(n^r)$ for any $s\geq s_0$.
\end{theorem}


\section{Auxiliary results and deduction of the main result}

In order to prove our main result, we consider a slightly stronger notion of induced Tur\'an numbers.
Given a graph $G$ and a partition of $V(G)$ into two sets $X,Y$, let $G[X,Y]$ denote the spanning bipartite subgraph
of $G$ with parts $X$ and $Y$  consisting of all the edges of $G$ between $X$ and $Y$.
Given a positive integer $s$ and a bipartite graph $H$, let $\ex^*_{\bip}(n,H,s)$ denote the minimum $M$ such
that for any $K_{s,s}$-free $n$-vertex graph $G$ and any partition of $V(G)$ into $X,Y$ if $e(G[X,Y])>M$ then 
$G[X,Y]$ contains a copy of $H$ that is an induced subgraph of $G$. Observe briefly that $\ex_{\bip}^*(n, H, s) \geq \frac{1}{2}\ex^*(n, H, s)$, so giving an upper bound on $\ex_{\bip}^*(n, H, s)$ gives an upper bound on $\ex^*(n, H, s)$. 

Next, we mention a few definitions originally used by Bukh and Conlon \cite{BC}
which were extended by Kang, Kim, Liu \cite{KKL}.
Let $F$ be a graph and $R$ a proper subset of $V(F)$, called the {\it root set}.
For each nonempty subset $S\subseteq V(F)$, let $e_S$ denote the number of edges of $F$ incident to a vertex in $S$ and let $\rho_F(S)=\frac{e_S}{|S|}$.
Let $\rho(F)=\rho_F(V(F)\setminus R)$. We say that $(F,R)$ (or $F$ if $R$ is clear) is {\it balanced} if $\rho_F(S)\geq \rho(F)$ holds for each nonempty $S\subseteq V(F)\setminus R$. 

Let $F_R^\ell$ denote the graph consisting of $\ell$ labeled copies $F_1,\dots, F_\ell$ of $F$ that
such that for each $v\in R$, the images of $v$ in $F_1,\dots F_\ell$ are the same
and that $F_1,\dots, F_\ell$ are pairwise vertex disjoint outside the common image of $R$.
We call $F_R^\ell$ the {\it $\ell$-th power} of $F$ {\it rooted at $R$}.
When the context is clear, we will drop the subscript $R$. 

Bukh and Conlon \cite{BC} proved the following celebrated result.
\begin{lemma} [\cite{BC}] \label{lem:BC-lower}
For every balanced rooted bipartite graph $(F,R)$ with $\rho(F)>0$, there exists 
a positive integer $\ell_0=\ell_0(F)$ such that for all $\ell\geq \ell_0$, $\ex(n,F_R^\ell)=\Omega(n^{2-\frac{1}{\rho(F)}})$.
\end{lemma}
As noted by Kang, Kim, and Liu \cite{KKL}, in the original statement of Bukh and Conlon, $F$ is a tree and $R$ is an independent set, but the statement still holds without these
extra assumptions. 

Let $r,t$ be positive integers. Let $T_{r,t}$ denote the height two tree obtained from an $r$-star $S$ by joining $t$ leaves to each
leaf of $S$. Let $R$ be the set of leaves of $T_{r,t}$. For each positive integer $\ell$, let $F_{r,t}^\ell=[T_{r,t}]^\ell_R$.
 We prove the following induced analogue of 
Theorem 1.5 of Conlon and Janzer \cite{CJ} (see \cite{JJM} for a weaker version of the theorem). This is the most important ingredient of our proof of the main result.

\begin{theorem} \label{thm:height-two}
Let $\ell,r,s,t$ be positive integers, where $r\geq t+2\geq 3$. Then $\ex^*_{\bip}(n,F^\ell_{r,t},s)=O(n^{2-\frac{r+1}{rt+r}})$.
\end{theorem}

We also need the following additional result.
\begin{theorem} \label{thm:asminus1}
Let $\ell,s\geq2$ be positive integers.
Let $T_{r,1,1}$ denote the tree obtained by adding a leaf to the central vertex of $T_{r,1}$. Let $R$ be the set of leaves of $T_{r,1,1}$. Let $H=[T_{r,1,1}]_R^\ell$. Then $\ex^*_{\bip}(n,H,s)=O(n^{2-\frac{r+1}{2r+1}})$.
\end{theorem}

Note that Theorem~\ref{thm:height-two} and Theorem~\ref{thm:asminus1} prove Conjecture~\ref{conj:HMST} for $F^\ell_{r, t}$, for $r \geq t+2$, and $T_{r, 1, 1}^{\ell}$ when $\ell$ is sufficiently large. 

The following two results follow from the proof of Theorem 1.2 of \cite{HMST}
and the proof of Theorem 1.5 of \cite{DGLL}, respectively.

\begin{theorem} [\cite{HMST}] \label{thm:one-side-bounded}
Let $H$ be a bipartite graph in which each vertex in one part has degree at most $k$, then for any positive integer $s$, $\ex^*_{\bip}(n,H,s)=O(n^{2-1/k})$.
\end{theorem}

\begin{theorem} [\cite{DGLL}] \label{thm:theta}
Let $\ell,s,t\geq 2$ be positive integers. Let $\Theta_\ell^t$ be the union
of $t$ internally disjoint paths of length $\ell$ with the same endpoints. Then $\ex^*_{\bip}(n,\Theta_\ell^t, s)=O(n^{1+1/\ell})$.
\end{theorem}

To prove Theorem \ref{thm:main}, we will apply the following reduction theorem to 
Theorem \ref{thm:height-two}, Theorem \ref{thm:asminus1}, Theorem \ref{thm:one-side-bounded} and Theorem \ref{thm:theta}.
This reduction lemma  may be viewed as an induced version of a well-known reduction theorem of Erd\H{o}s and Simonovits \cite{ES-regular}.

\begin{definition}
Given a bipartite graph $F$ with parts $A, B$, we let $F(t)$ denote the graph formed by taking the disjoint union of a copy of $K_{t, t}$ with parts $C, D$ and $F$ and 
adding all the edges between $A$ and $D$ and all the edges between $B$ and $C$.
\end{definition}

\begin{theorem} [Reduction theorem] \label{thm:reduction}
Let $s$ be a positive integer and $0<\alpha<1$ be a real. 
Let $F$ be a graph such that $ \ex^*_\bip(n, F, s) = O_s(n^{2 - \al })$. Then $\ex^*_\bip(n, F(1), s) = O_s(n^{ 2- \frac{\alpha}{\alpha+1}})$.
\end{theorem}

Now, we show how Theorem \ref{thm:main} follows from the results mentioned in this section. This recursive argument is as in those by Kang, Kim, and Liu
\cite{KKL} and by Conlon and Janzer \cite{CJ} for $\ex(n,H)$.

\begin{proof} [Proof of Theorem \ref{thm:main} using Lemma \ref{lem:BC-lower},
and Theorems \ref{thm:height-two}, \ref{thm:asminus1}, \ref{thm:one-side-bounded}, \ref{thm:theta}, \ref{thm:reduction}]
We call a rational $r$ {\it realisable} if there exist a graph $H$
and a constant $s_0$ such that $\ex^*(n,H,s)=\Theta(n^r)$ for all $s\geq s_0$.
We call $r=2-\frac{a}{b}$, where $a,b$ are positive integers satisfying $b>a$, {\it good}
if there is a balanced rooted bipartite graph $(F,R)$ satisfying that
$\rho(F)=\frac{b}{a}$ and that $\ex^*_{\bip}(n,F^\ell_R,s)=O(n^{2-\frac{a}{b}})$ for all $\ell,s$. 
First, we claim that if $r$ is good then $r$ is realisable. Suppose $r$ is good with a corresponding $(F,R)$.
Then by Lemma \ref{lem:BC-lower}, there exists $\ell_0$ such that for all
$\ell\geq \ell_0$, $\ex(n,F_R^\ell)=\Omega(n^{2-\frac{a}{b}})$.
Let $s_0=|V(F_R^\ell)|$. Then clearly $\ex^*(n,H,s)\geq \ex(n,F_R^\ell)\geq\Omega(n^{2-\frac{a}{b}})$ for all $s\geq s_0$.
On the other hand, $\ex^*(n,F_R^\ell,s)\leq 2\ex^*_{\bip}(n,F_R^\ell,s)\leq 
O(n^{2-\frac{a}{b}})$. Hence $\ex^*(n,F_R^\ell,s)=\Theta(n^r)$ for all $\ell\geq \ell_0, s\geq s_0$. Thus, $r$ is realisable.
Hence, it suffices to prove that for all $r=2-\frac{a}{b}$, where $a,b$ are positive integers satisfying $b\geq \max\{a, (a-1)^2\}$, $r$ is good.

Now, suppose $r=2 - \frac{a}{b}$ is good with a corresponding $(F,R)$. Let $F'=F(1)$ and $R'$ be the union of $R$ and the two vertices added to $F$ to form $F(1)$.
It is easy to check that $(F',R')$ is a balanced rooted bipartite graph with $\rho(F')=\rho(F)+1=\frac{b}{a}+1=\frac{a+b}{a}$. Let $\ell,s$ be any positive integers.
Note that $[F']^\ell_{R'}=F^\ell_R(1)$. Since $r$ is good, $\ex^*_{\bip}(n,F^\ell_R, s)=O(n^{2-\frac{a}{b}})$.
By Theorem \ref{thm:reduction}, $\ex^*_{\bip}(n,F^\ell_R(1), s)=O(n^{2-\frac{a}{a+b}})$. In other words, $\ex^*_{\bip}(n,(F')^\ell_{R'},s)=O(n^{2-\frac{a}{a+b}})$.
Hence, $2-\frac{a}{a+b}$ is also good. By iterating the above argument, we see that whenever $2-\frac{a}{ap+q}$ is good for $p=p_0$, it is good for all integers $p\geq p_0$.

By Theorem \ref{thm:height-two}, $2 - \frac{r + 1}{rt + r} = 2 - \frac{r + 1}{(r + 1)t + r - t}$ is good for all $r\geq t+2\geq 3$. Hence,
$2 - \frac{r + 1}{ (r + 1)p + (r - t)}$ is good for all $p \geq t$. Since $r - t$ ranges from $2$ to $r -1$, for $r \geq 3, d \geq r^2$ with $d \neq -1, 0, 1 \pmod{r + 1}$ we get rationals of the form $2 - \frac{r + 1}{d}$. If $d \equiv 0 \pmod{r +1}$, we see rationals of the form $2 - \frac{1}{t}$, which are given by $K_{t, \ell}$ \cite{HMST}. If $d \equiv 1 \pmod{r + 1}$, we note that the exponent associated with $\Theta_{r + 2}^t $ is $2 - \frac{r + 1}{ r + 2}$, so iteratively applying Theorem~\ref{thm:reduction} to $\Theta_{r + 2}^t$ gives all exponents of the form $2 - \frac{a}{pa + 1}$ for $p \geq 1$. Similarly, we note the exponent of $T_{r, 1, 1}^{\ell}$ is $2 - \frac{r + 1}{2(r + 1) - 1}$, iteratively applying Theorem~\ref{thm:reduction} gives us all exponents of the form $2 - \frac{a}{pa - 1}$ for $p \geq 2$. Since for $r = 1, 2$, every $d$ is either equivalent to $0, 1, -1 \pmod{r + 1}$, the proof is completed. 

\end{proof}

Hence, to complete the proof, it remains to prove Theorem \ref{thm:reduction}, Theorem \ref{thm:height-two}
and Theorem \ref{thm:asminus1}.


\section{Notation and regularization}
Given a graph $G$, we use $\Delta(G), \delta(G)$ to denote the maximum and minimum degree of $G$, respectively.
Let $K\geq 1$ be a real. We say that a graph $G$ is {\it $K$-almost regular} if $\Delta(G)\leq K \delta(G)$.
Given a vertex $v$ in $G$, let $N_G(v)$ denote the {\it neighborhood} of $v$ in $G$. 
Given a set $S$ of vertices, let $N^*_G(S)=\bigcap_{v\in S} N_G(v)$ denote the {\it common neighborhood} of $S$ in $G$.
When the context is clear, we drop the subscript.

In this paper, we frequently use a standard tool to reduce a dense host graph to an almost regular induced subgraph with similar relative density. 
Such a reduction tool was first developed by Erd\H{o}s and Simonovits \cite{ES-regular}
and there have been many variants of it (see for instance \cite{Jiang-Seiver},
\cite{CJL}, \cite{HMST} and etc). For convenience, we will use the most recent version of it, given by Dong, Gao, Li, and Liu \cite{DGLL}.

\begin{lemma} [\cite{DGLL} ] \label{lem:almost-reg} 
Suppose $\alpha, C>0$ and let $K=2^{\frac{4}{\alpha}+2}$. Let $G$ be an $n$-vertex graph on $n$ vertices with $e(G)\geq Cn^{1+\alpha}$. Then there exists a 
$K$-almost-regular induced subgraph $H\subseteq G$ on $m$ vertices such
that $e(H)\geq \frac{C}{4} m^{1+\alpha}$ and $m\geq \frac{C^{\frac{\alpha+1}{2\alpha+4}}}{K} n^{\frac{\alpha}{2\alpha+4}}$. 
\end{lemma}

\section{Preliminary lemmas and Proof of Theorem \ref{thm:reduction}}

In this section, we develop some preliminary lemmas and prove Theorem \ref{thm:reduction}. 
We start with the following standard averaging lemma.
\begin{lemma} [\cite{Jiang-Qiu2} ] \label{lem:JQ-lem}
Let $0<c<1$ be a real and $s$ a positive integer. Let $G$ be a bipartite graph with a
bipartition $(X,Y)$. Suppose that $e(G)\geq c|X||Y|$ and that $c|X|\geq 2s$. Then there exists
an $s$-set in $X$ such that $|N^*_G(S)|\geq (c/2)^s|Y|$.
\end{lemma}

The next lemma follows immediately from Lemma \ref{lem:JQ-lem}
and will be frequently used in our proofs.

\begin{lemma} \label{lem:small-bad-set}
Let $0<c<1$ be a real. Let $s\geq 2$ be a positive integer.
Let $G$ be $K_{s,s}$-free graph. Let $W$ be a set of vertices with 
$|W|\geq s(\frac{2}{c})^s$.
Let $B(W)$ be the set of vertices in $G$ outside $W$ that have at least $c|W|$ neighbors in $W$.
Then $|B(W)|<2s/c$.
\end{lemma}
\begin{proof}
Let $B=B(W)$. 
Let $H$ be a bipartite subgraph of $G$ with parts $B$ and $W$ consisting of all
the edges of $G$ between $B$ and $W$. 
Suppose for contradiction that $|B|\geq 2s/c$.
Then $e(H)\geq c|B||W|$ and $c|B|\geq 2s$.
By Lemma \ref{lem:JQ-lem}, there exists an $s$-set $S$ in $B$ such that 
$|N^*_H(S)|\geq (c/2)^s |W|\geq s$, contradicting $G$ being $K_{s,s}$-free.
Hence, we must have  $|B(W)|<2s/c$.\end{proof}

The following theorem was established in \cite{DGLL}.

\begin{theorem}[\cite{DGLL} ] \label{thm:DGLL-theorem}
Let $T$ be a $t$-vertex tree. If $G$ is an $n$-vertex $K$-almost-regular $K_{s,s}$-free graph with average degree
$d\geq (4Kt)^{6s}s^3$, then there are at least $n(\frac{d}{2K})^{t-1}$ labeled induced copies of $T$ in $G$.
\end{theorem}
 We need a slightly more general version of this result stated as below. This version would
 follow from essentially the same proof as Theorem \ref{thm:DGLL-theorem}. However, for completeness, we give a self-contained short proof of it using Lemma \ref{lem:small-bad-set}.
In our lemma, $d$ denotes the minimum degree.

\begin{lemma}\label{lem:countingtrees}
Let $T$ be a $t$-vertex tree. Let $G$ be an $n$-vertex $K_{s,s}$-free graph.
Suppose $L$ is $K$-almost-regular subgraph of $G$ with
minimum degree at least $d\geq st^22^{t + 6} K^{t - 1}$. Then $L$ contains
at least $n(\frac{d}{2})^{t-1}$ copies of $T$ which are induced in $G$.
\end{lemma}

\begin{proof}
For each vertex $x\in V(L)$, let $B(x)=\{y\in V(L): |N_G(y)\cap N_L(x)|\geq \frac{1}{4t} d\}$.
Since  $d\leq  |N_L(x)| \leq Kd$ and 
 $d\geq s(8tK)^s$, applying Lemma~\ref{lem:small-bad-set} with $c=\frac{1}{4Kt}$ we have $|B(x)| \leq 8stK$.
For convenience, we will call a tree $S$ on at most $t$ vertices in $L$ {\it good} if it is an induced subgraph of $G$ and for 
any $x\in V(S)$ and $y\in V(S)\setminus \{x\}$, $y\notin B(x)$. Let $v_1,\dots, v_t$ be an ordering of the vertices of $T$ such
that for each $i\in [t]$, $T_i=T[\{v_1,\dots, v_i\}]$ is a tree and $v_i$ is a leaf of $T_i$.
We use induction on $i$ to prove that for each $i\in [t]$, $L$ contains at least $n(\frac{d}{2})^{i-1}$ good copies of $T_i$. The basis step is trivial. Let $2\leq i\leq t$
and suppose the claim holds for $T_{i-1}$.

Let $u$ denote the unique neighbor of $v_i$ in $T_i$. Let $\Sc$ be the family of good copies of $T_{i-1}$ in $L$. 
By the induction hypothesis, $|\Sc|\geq n(\frac{d}{2})^{i-2}$.
Fix some $S \in \Sc$. 
Let $u'$ be the image of $u$ in $S$. Let $B =\bigcup_{x\in V(S)} B(x)$. By our earlier discussion, $|B| \leq 8st^2K<\frac{1}{8}d$.
Let $\Gamma(S) = N_L(u') \setminus (\bigcup_{x \in V(S)} N_G(x) \cup B)$. Since $S\in \Sc$, for each $x\in V(S)$, $|N_G(x)\cap N_L(u')|\leq \frac{1}{4t} d$. Hence, $|N_L(u')\setminus \bigcup_{x\in V(S)} N_G(x)|\geq \frac{3}{4}d$
and $|\Gamma(S)|\geq \frac{5}{8}d$. Note that for each $v'\in \Gamma(S)$, $S' = S \cup u'v'$ is a copy  of $T_i$ in $L$ that is induced in $G$. Also, for all $x$ in $S$ and $y \in S', y \neq x, y\notin B(x)$. Let $\T = \{ S \cup u'v': S \in \Sc \text{ and } v' \in \Gamma(S)\}$. 
Note $|\T| \geq |\Sc|\frac{5}{8}d\geq \frac{5}{4}n(\frac{d}{2})^{i-1}$. Note that a member $S'$ of $\T$ would be good
unless it contains a vertex in $B(v')$, where $v'$ is the image of $v_i$ in $S'$. Let us call such a member a {\it bad} member.

Since there are at most $n$ choices for $v'$, for each $v'$ we have $|B(v')|\leq 8stK$ and 
$L$ has maximum degree at most $Kd$, the number of bad members of $\T$ is at most $n(8st^2K) (Kd)^{i-2}$.
Using $d \geq st^2 2^{t + 4} K^{t - 1}$, we see that the number of bad members of $\T$ is no more than $\frac{1}{4}n(\frac{d}{2})^{i-1}$. 
Thus, there are at least $n(\frac{d}{2})^{i-1}$ good members of $\cH$. This completes the induction and the proof.
\end{proof}

Now, we are ready to prove Theorem \ref{thm:reduction}
\begin{proof}[Proof of Theorem \ref{thm:reduction}]
Let $G$ be an $n$-vertex $K_{s,s}$-free graph with a partition $X,Y$ of $V(G)$ such that
 $e(G[X,Y])\geq Cn^{2-\frac{\alpha}{\alpha+1}}$, where $C$ is sufficiently large. 
 Suppose for contradiction that $G[X,Y]$ does not contain a copy of $F(1)$ that is induced in $G$. Let $K=2^{4(\alpha+1)+2}$. By  Lemma \ref{lem:almost-reg},
there exists a $K$-almost-regular induced subgraph $L\subseteq G[X,Y]$ on $m$ vertices such that $e(L)\geq \frac{C}{4} m^{2-\frac{\alpha}{\alpha+1}}$ and 
$m\geq \Omega (n^{\frac{\alpha}{6\alpha}+4})$.
Then $L$ has minimum degree  $d\geq \frac{C}{2K} m^{\frac{1}{\alpha+1}}$.
Let $X'=X\cap V(L)$ and $Y'=Y\cap V(L)$. Let $G'=G[V(L)]$. Note that
$L=G'[X',Y']$. For convenience, we call any subgraph of $L$ {\it good} if 
it is an induced subgraph of $G'$ (and hence of $G$).
Let $\F$ denote the family of good $K_{1,2}$'s in $L$ that are centered in
$X'$. By Lemma \ref{lem:countingtrees}, without loss of generality, $|\F|\geq \frac{1}{8}m d^2$. 

Since $G$ is $K_{s,s}$-free, for
any set $T$ of leaves of a member of $\F$ we have that every set of $2s$ vertices contained 
in $N^*_L(T)$ has a subset $T'$ of $2$ vertices that are nonadjacent in $G$,
so that $T \cup T'$ induces  $K_{2, 2}$ in $G$. Thus, we have that the number of good
$K_{2, 2}$'s  in $L$ with one part being $T$ is at least 
$\frac{1}{\binom{2s}{2}}\binom{|N^*_L(T)|}{2}$.  Hence, summing over all such $T$
(for which there are most $\binom{m}{2}$ of them), we see that
the number of good $K_{2, 2}$'s in $L$ is
at least 
\[\sum_T \frac{1}{\binom{2s}{2}}\binom{|N^*_L(T)|}{2} 
\geq \frac{1}{\binom{2s}{2}}\binom{m}{2} \left(\frac{|\F|}{2\binom{m}{2}}\right)^2
\geq \frac{1}{128\binom{2s}{2}}d^4.\] 

By the pigeonhole principle, there exists and edge
$xy$ in $L$ such that the number of good $K_{2,2}$'s in $L$ containing $xy$
is at least $\frac{1}{128\binom{2s}{2}}\frac{d^3}{Km}$. Let $A=N_H(x)\setminus N_G(y)$ and $B=N_H(y)\setminus N_G(x)$. Since $L$ has maximum degree at most $Kd$, $|A|,|B|\leq Kd$.
Note that each good $K_{2,2}$ in $L$ containing $xy$ corresponds to a distinct edge
of $L[A,B]$. Hence, $e(L[A, B]) \geq \frac{1}{128\binom{2s}{2}}\frac{d^3}{Km}$. 

On the other hand, observe that $L[A, B]$ cannot contain an induced copy of $F$,
because the vertex set of such a copy together with $\{x,y\}$ would induce a copy
of $F(1)$  in $G$, contradicting $G$ not containing an induced $F(1)$.
Hence,
$e(L[A, B]) \leq \ex_\bip^*(2Kd, F, s) = O_s((Kd)^{2 - \alpha})$. Combining this with our lower bound on $e(L[A,B])$, we get $d^3m^{-1}=O_{\alpha,s}(d^{2-\alpha})$. 
Hence, $d=O_{\alpha,s}(m^{\frac{1}{\alpha+1}})$. By choosing $C$ to be sufficiently large, 
this contradicts our earlier claim that $d\geq \frac{C}{2K}m^{\frac{1}{\alpha+1}}$.
This completes the proof. 
\end{proof}

\section{More embedding lemmas}
In this section, we develop some useful lemmas for embedding induced subgraphs
into a $K_{s,s}$-free graph. 
First, we recall the  K\H{o}v\'ari-S\'os-Tur\'an theorem \cite{KST}.

\begin{lemma} [\cite{KST}] \label{lem:KST}
Let $s$ be a positive integer. If $H$ is a $K_{s,s}$-free bipartite graph with parts $X,Y$, each of size $m$.
Then $e(H)\leq (s-1)^{1/s} m^{2-1/s}+(s-1)m$.
\end{lemma}

The next technical lemma builds on several ideas from \cite{HMST} and \cite{JP} and provides one of the main ingredients of our arguments.
Given a bipartite graph $H$ with ordered partition $(A,B)$, the {\it neighborhood hypergraph} of $H$ induced on $A$, denoted
by $\F_A(H)$ is the  multi-hypergraph whose edge set is the multi-set $\{N_H(y): y\in Y\}$. A hypergraph $\F$ is called {\it $k$-partite} if
its vertices can be partitioned into $k$ parts so that each hyperedge contains at most one vertex in each part.

If $\F$ is a $k$-uniform hypergraph and $m$ is a positive integer, then the {\it $m$-blowup} of $\F$, denoted by $\F[m]$ is the hypergraph obtained as follows.
First, we replace each vertex $v$ of $\F$ with a set $S_v=\{v^{(1)},\dots, v^{(m)}\}$ of $m$ vertices so that the $S_v$'s are pairwise disjoint.
Then we replace each edge $e$ of $F$ with the complete $k$-uniform $k$-partite hypergraph with parts $\{S_v: v\in e\}$. We will call the $S_v$'s the {\it parts} of $\F[m]$.

\begin{lemma}[Key lemma] \label{lem:rich-independent}
Let $s\geq 2$ be an integer.  Let $H$ be a bipartite graph with $h$ vertices with an ordered partition $(A,B)$.
Let $m=2^{ h + s + 3}s^s h^{2s}$ and $C(H,s)= s(4h)^{s + 1}$.
Let $G$ be a $K_{s,s}$-free graph and $L$ a bipartite subgraph of $G$ with parts $X,Y$.
Suppose    $\D:=\{S\subseteq X: |N^*_L(S)|\geq C(H,s)\}$ contains the $m$-blowup of $\F_A(H)$. 
Then $L$ contains a copy of $H$ that is an induced subgraph of $G$.
\end{lemma}
\begin{proof}
For each $W\subseteq V(G)$, let $B(W)=\{x\in V(G)\setminus W: |N_G(x)\cap W|\geq (1/2h)|W|\}$.
Since $G$ is $K_{s,s}$-free, by Lemma \ref{lem:small-bad-set}, we have 
\begin{equation} \label{eq:bound-on-BW}
\forall W\subseteq V(G), \mbox{ if } |W|\geq s(4h) ^s, \mbox{ then } |B(W)|\leq 4sh.
\end{equation}
Recall that $\F_A(H)$ is the neighbhorhood multi-hypergraph of $H$ induced on $A$. Suppose $\D$ contains the $m$-blowup $\F^*$ of $\F_A(H)$, with parts $\{S_v: v\in A\}$. Let $\phi$ be a 
random mapping from $A$ to $\bigcup_{v\in A} S(v)$ where each $v\in A$ is mapped to a vertex
in $S_v$ uniformly at random.
For each $e\in \F_A(H)$, let $\phi(e)=\{\phi(v): v\in e\}$. Consider any $e\in \F_A(H)$.
Since $\phi(e)\in \D$, by definition, we have $|N^*_L(\phi(e))|\geq C(H,s)\geq s(4h)^s$. By \eqref{eq:bound-on-BW},  $|B(N^*_L(\phi(e)))|\leq 4sh$. For any $u\in A, u\notin e$, $\mathbb{P}(\phi(u)\in B(N^*_L(\phi(e)))\leq 4sh/m$. Thus, $\mathbb{P}(\exists u\in A, u\notin e, 
\phi(u)\in B(N^*_L(\phi(e))))<4sh^2/m\leq 1/2^{h+1}$, by our choice of $m$. Since $\F_A(H)$
has at most $2^h$ hyperedges, by the union bound, the probability that there exists $e\in \F_A(H)$ such that there exists $u\in A, u\notin e, \phi(u)\in B(N^*_L(\phi(e)))$ is less than $1/2$.

Consider now any two vertices $u,v\in A$. Since $G$ is $K_{s,s}$-free, by Lemma \ref{lem:KST}, the number of edges of $G$ between $S_u$ and $S_v$ is at
most $(s-1)^{1/s} m^{2-1/s}+(s-1)m\leq 2sm^{2-1/s}$. Hence, the probability of $\phi(u)\phi(v)\in E(G)\leq 2sm^{-1/s}$. By our choice of $m$, we have
\[\mathbb{P}(\exists u,v\in A \phi(u)\phi(v)\in E(G))\leq \binom{h}{2} 2s m^{-1/s}<\frac{1}{2}.\]
Thus, there exists a choice of $\phi$ such that
\begin{enumerate}
\item $\{\phi(v): v\in A\}$ is an independent set in $G$.
\item for each $e\in \F_A(H)$ and for each $u\in A,
u\notin e$, $\phi(u)\notin B(N^*_L(\phi(e)))$.
\end{enumerate}

 Fix such a choice of $\phi$.
 For each $e\in \F_A(H)$, let $\Gamma(e)=N^*_L(\phi(e))\setminus \bigcup \{N_G(\phi(u)): u\in A, u\notin e\}$. Let $t = 2 s (4h)^s.$
 Since $\phi$ satisfies condition 2 and $C(H,s)=s(4h)^{s+1}$, we have
 \begin{equation} \label{lem:induced-neigbhorhood}
 \forall e\in \F_A(H), |\Gamma(e)|\geq 
(1-h\cdot\frac{1}{2h}) |N^*_L(\phi(e))|\geq \frac{C(H,s)}{2}= ht.
\end{equation}
Suppose $B=\{b_1,\dots, b_q\}$. For each $i\in [q]$, let $e_i=N_H(b_i)$. Then
$\F_A(H)=\{e_1,\dots, e_q\}$. Since $|\F_A(H)|=q$ and for each $e\in \F_A(H), |\Gamma(e)| \geq qt$, by Hall's theorem, there exist
pairwise disjont $t$-sets $U_1,\dots, U_q$ such that for each $i\in [q]$ $U_i\subseteq \Gamma(e_i)$.
By definition, for each $i\in [q]$ and each $y\in U_i$, $N_L(y)\cap \phi(A)=N_G(y)\cap \phi(A)=\phi(e_i)
=\phi(N_H(b_i))$. We show that $\phi$ can be extended to an injection of $A\cup B$ into
$V(G)$ such that $L[\phi(A\cup B)]=G[\phi(A\cup B)]\cong H$. 

We will embed $b_1,\dots, b_q$ in that order. We use induction to show that
for each $i=1,\dots, q$, we can embed $b_i$ into $U_i$ as $\phi(b_i)$ such that $\phi(b_1),\dots, \phi(b_i)$ are pairwise nonadjacent in $G$ and for each $j=i+1,\dots, q$,
$|U_j\setminus \bigcup_{\ell=1}^i N_G(\phi(b_\ell))|\geq (1-1/2h)^i |U_j|$.
For the basis step, for $j=2,\dots, q$, since $|U_j|\geq t \geq  s (4h)^s$, by \eqref{eq:bound-on-BW}, we have $|B(U_j)|\leq 4sh$. Hence $|\bigcup_{j=2}^q B(U_j)|\leq 4sh^2<|U_1|$. Let $\phi(b_1)$
be any vertex of $U_1\setminus \bigcup_{j=2}^q B(U_j)$. Then for each $j=2,\dots, q$,
$|U_j\setminus N_G(\phi(b_1))|\geq (1-1/2h)|U_j|$. For the induction step, let $1\leq i\leq q-1$
and suppose we have found $\phi(b_1),\dots \phi(b_i)$ that satisfy the conditions.
For $j=i+1,\dots, q$, let $U'_j=U_j\setminus \bigcup_{\ell=1}^i N_G(\phi(b_\ell))$.
By the induction hypothesis, $|U'_j|\geq (1-1/2h)^i|U_j|\geq (1-1/2h)^it\geq s(4h)^s$, since $(1 - 1/2h)^i \geq \frac{1}{2}$. By \eqref{eq:bound-on-BW},
$|\bigcup_{j=i+2}^q B(U'_j)|\leq q(4sh)<4sh^2<|U'_{i+1}|$. Let $\phi(b_{i+1})$ be any
vertex in $U'_{i+1}\setminus \bigcup_{j=i+2}^q B(U'_j)$. Since $\phi(b_{i + 1}) \not \in \bigcup_{j=i+2}^q B(U'_j)$, 
for $j=i+2,\dots, q$, we have 
$|U_j\setminus \bigcup_{\ell=1}^{i+1} N_G(\phi(b_\ell)|\geq (1-1/2h)|U'_j|\geq (1-1/2h)^{i+1}|U_j|$.
This completes the induction. By our procedure, $\{\phi(b_1),\dots, \phi(b_q)\}$ is independent in $G$ and
for each $i\in [q]$, $N_L(\phi(b_i)\cap \phi(A)=N_G(\phi(b_i))\cap \phi(A)=\phi(N_H(b_i))$. So, $L[\phi(A\cup B)]=G[\phi(A\cup B)]\cong H$, as desired.
 \end{proof}

\begin{proposition} \label{prop:heavy-stars}
Let $K\geq 1$ be a real.
Let $h,p,s$ be positive integers. Let $H$ be a bipartite graph with $h$ vertices and an ordered bipartition $(A,B)$ such that the neighborhood hypergraph $\F_A(H)$ of $H$ induced by $A$ is $p$-partite. Let $C(H,s)=s(4h)^{s+1}$.
For any $\varepsilon>0$, there exists a constant
$C_1=C_1(\varepsilon,h, p, s, K)$  such that the following holds.
Let $G$ be a $K_{s,s}$-free graph on $n$ vertices. Let $L$ be a 
$K$-almost-regular bipartite subgraph of $G$ with minimum degree $d\geq C_1$.
Suppose $L$ does not contain a copy of $H$ that is induced in $G$. Let $(X,Y)$ be a bipartition of $L$. 
Then the number of $p$-stars in $L$ with a leaf set $S$ satisfying $|N^*_L(S)|\geq C(H,s)$  is at most $\varepsilon nd^p$.
\end{proposition}
\begin{proof}
Let $m=2^{h+s+3}s^sh^{2s}$ as in Lemma \ref{lem:rich-independent}. Let $(X,Y)$ be a bipartition of $L$.
Consider any vertex $v$. Say $v\in Y$. Since $\F_A(H)[m]$ is $p$-partite, by a theorem of Erd\H{o}s \cite{Erdos}, $\ex(d,\F_A(H)[m])=o(d^p)$.
By choosing $C_1$ to be large enough, we have $\ex(|N_L(v)|,\F_A(H)[m])\leq (\varepsilon/K^p)\binom{|N_L(v)|}{p}$.
Let $\D$ be the $p$-uniform hypergraph on $N_L(v)\subseteq X$ such that a $p$-set $S$ in $N_L(v)$ is an edge in $\D$
if and only if $|N^*_L(S)|\geq C(H,s)$. If $|\D|>\varepsilon d^p$, then we have
$|\D|>(\varepsilon/K^p)\binom{|N_L(v)|}{p} \geq \ex(|N_L(v)|, \F_A(H)[m])$. So, $\D$ contains a copy of $\F_A(H)[m]$.
By Lemma \ref{lem:rich-independent}, $L$ contains a copy of $H$ that is induced in $G$, a contradiction. Hence, $|\D|\leq \varepsilon d^p$.
Since this holds for any vertex $v$, the claim follows.
\end{proof}

\begin{proposition} \label{prop:asymmetric}
 Let $G$ be a $K_{s,s}$-free graph. Let $H$ be a one-side $p$-bounded bipartite graph. There exists constant $C_2, C_3$ depending only on $H, s$ such that
the following holds.
Let $M\subseteq G$ be a bipartite subgraph with parts $X,Y$, such that  $d_M(y) \geq \delta_Y \geq C_2$. If $e(M) \delta_Y^{p - 1} \geq C_3|X|^p$, then $M$ contains a copy of $H$ that is induced in $G$. 
\end{proposition}
\begin{proof}
 Let $\pi$ be the Tur\'an density of $\F_A(H)$, and set $\gamma = \frac{\pi + 1}{2}$. Let $m$ be as given in Lemma~\ref{lem:rich-independent}, and $C_2$ be sufficiently large so that 
 \begin{equation} \label{eq:turan-density}
 \forall n\geq C_2,
 \ex(n, \F_A(H)[m]) \leq \gamma \binom{n}{p}.
 \end{equation}
 Note that since $\gamma > \pi$, such a choice of $C_2$ exists. Let $C_3 = (1 - \gamma)^{- 1} sp^p(4h)^{s + 1}  $.

By our conditions, $e(M)\geq \delta_Y |Y|$ and $e(M) \delta_Y^{p - 1} \geq C_3|X|^p$. Hence,
\[[e(M)]^{p} \geq [e(M)]\delta_Y^{p-1}|Y|^{p - 1}\geq  (1 - \gamma)^{- 1}s p^p(4h)^{s + 1} |X|^{p}|Y|^{p - 1}.\] 
Let $y$ be a uniformly random vertex from $Y$ and let $T = N(y)$. Clearly, 

$$\mathbb{E}\left[ \binom{|T|}{ p } \right] \geq \binom{e(M)/|Y|}{p } \geq \frac{[e(M)]^{p}}{p^{p }|Y|^{p }} \geq \frac{(1 - \gamma)^{- 1} s(4h)^{s + 1}|X|^{ p}}{|Y|}.$$

Call a set $S$ of vertices in $X$ {\it bad} if $|N^*_M(S)|\leq s(4h)^{s + 1}$. Let $B$ denote the number of
bad $p$-sets in $T$. Then, $\mathbb{E}[B] \leq \frac{ s(4h)^{s + 1} |X|^{p} }{|Y|}$. 
Hence, $(1-\gamma) \mathbb{E}\left[ \binom{|T|}{ p} \right] \geq \mathbb{E}[B]$.
Hence, for some choice of $y$, we have $B\leq (1-\gamma){\binom{|T|}{p}}$. Fix such a $y$. By our choice, 
the number of $p$-sets $S$ in $T$ with $|N^*_M(S)| \geq  s(4h)^{s + 1}$ is greater than $\gamma \binom{|T|}{p}$.   Since $|T|=d_M(y)\geq C_2$, by \eqref{eq:turan-density}, there is an $m$-blowup of $\F_A(H)$ in $T$ with edges corresponding to $p$-sets $S$ with $|N^*_M(S)| \geq  s(4h)^{s + 1}$. Applying Lemma~\ref{lem:rich-independent}, the proof is complete. 
\end{proof}

Let $F$ be a graph and $R$ a proper subset of $V(F)$. Let $\ell$ be a positive integer. Recall that $F_R^\ell$ is
the graph consisting of $\ell$ labeled copies $F_1,\dots, F_\ell$ of $F$ that
such that for each $v\in R$, the images of $v$ in $F_1,\dots F_\ell$ are the same
and that $F_1,\dots, F_\ell$ are pairwise vertex disjoint outside the common image of $R$. If $G$ is a graph
then we call a copy of $F^\ell_R$ in $G$ {\it semi-induced} if the $F_i$'s are induced copies of $F$ in $G$.

\begin{lemma}  \label{lem:semi-induced-to-induced}
Let $\ell,s$ be positive integers.
Let $T$ be a tree. Let $R$ denote the set of leaves of $T$. 
There exists a constant $\lambda=\lambda(T,\ell)$ such that the following holds.
Let $G$ be a $K_{s,s}$-free graph. If $F^*$ is a semi-induced copy of $F^\lambda_R$ in $G$,
then $F^*$ contains an copy of $F^\ell_R$ that is induced in $G$.
\end{lemma}
\begin{proof}
Suppose $V(F)\setminus R=\{v_1,\dots, v_q\}$. By definition, $F^*$ consists of induced copies $F_1,\dots F_\lambda$ of $F$ in $G$
such that for each $v\in R$ the image of $v$ in the $F_i$'s are the same and that the $F_i$'s
are pairwise vertex disjoint outside the common image of $R$. For each $r\in [q], i\in [\lambda]$ let $v_r^{(i)}$ denote the image of
$v_r$ in $F_i$. Let $\lambda=R_{q^2+1}(\ell, 2s,\dots, 2s)$, i.e. the minimum $n$ such that every $(q^2+1)$-coloring of $E(K_n)$ 
yields either a monochromatic $K_\ell$ in the first color or a monochromatic $K_{2s}$ in another color.

Let $H$ be an edge colored graph on $\{F_1, \dots F_\lambda\}$ using $q^2$ colors defined as follows. 
For any $1\leq i<j\leq \lambda$, if $G$ contains some edge between $V(F_i)$ and $V(F_j)$, we place an edge between $F_i$ and $F_j$, and then we color the edge in $H$ with color $(a, b)$ if the edge in $G$ is between $v^{(i)}_a$ and $v^{(j)}_b$. If there are many colors we could pick, we fix one arbitrarily. 

We note that $H$ has no monochromatic clique of size $2s$. Indeed, suppose it has a clique of some color $(a,b)$ on vertices $F_{t_1}, F_{t_2},
\dots, F_{t_{2s}}$. Then $v_a^{i}v_b^{j}\in E(G)$ for any $i\in\{t_1,\dots, t_s\}$ and $j\in \{t_{s+1},\dots, t_{2s}\}$, contradicting $G$ being
$K_{s,s}$-free. 
By our choice of $\lambda$, $H$ must contain an independent set of size $\ell$. These $\ell$ corresponding $F_i$'s together form an induced copy of $F_R^\ell$ in $G$.
\end{proof}


\section{Proof of Theorem \ref{thm:height-two}}

In this section, we prove Theorem \ref{thm:height-two}. This is the most technical section of the paper.
While we use the framework of Conlon and Janzer \cite{CJ}, since we are embedding induced subgraphs, things need
to be set up more delicately in order for us to use the induced embedding tools developed in the last couple of sections.
Throughout the section, we let $T=T_{r,t}$. Let $x$ denote the central vertex of $T$ and we view it as the root of $T$.
Let $y_1,\dots, y_r$ denote the children of $x$. For each $i\in [r]$, let $z_{i,1},\dots, z_{i,t}$ denote the children of $y_i$. Let $R$ denote the set of leaves of $T$. 
Let $\ell$ be given. Let $H=T_R^\ell$.
Let $A$ be a sufficiently large constant depending on $r,t,\ell,s$. Let $G$ be an $n$-vertex $K_{s,s}$-free graph. Let $X,Y$ be a partition
of $V(G)$ and assume that $e(G[X,Y])\geq A n^{2-\frac{r+1}{rt+r}}$. We show that if $A$ is taken to be sufficiently large then
$G[X,Y]$ must contain a copy of $H$ that is induced in $G$. 

Suppose for contradiction that $G[X,Y]$ does not contain a copy of $H$ that is
induced in $G$. Let $\alpha=1-\frac{r+1}{rt+r}=\frac{rt-1}{rt+r}$ and let $K=2^{\frac{4}{\alpha}+2}$ as in Lemma \ref{lem:almost-reg},
By the lemma, $G[X,Y]$ contains an $m$-vertex $K$-almost-regular induced subgraph $L$ with $e(L)\geq \frac{A}{4}m^{1+\alpha}$, where $m=\Omega(n^{\frac{\alpha}{2\alpha+4}})$. Let $X_L=X\cap V(L)$ and $Y_L=Y\cap V(L)$. Let $G'=G[V(L)]$. Note that $G'[X_L,Y_L]=L$.
Let $d$ denote the minimum degree of $L$. By our assumption, 
\begin{equation}\label{eq:L-assumptions}
d\geq \frac{A}{2K}m^\alpha=\frac{A}{2K}m^{\frac{rt-1}{rt+r}} \mbox{ and }\Delta(L)\leq Kd.
\end{equation}
Let $C(H,s)=s(4h)^{s+1}$ as in Proposition \ref{prop:heavy-stars}. Let $\ve = \left(\frac{1}{4K}\right)^{2t + r}$, and $C_1$ be as in Proposition~\ref{prop:heavy-stars}. Note that for $A$ sufficiently large, $d \geq C_1$, which allows us to apply Proposition~\ref{prop:heavy-stars} to $L$ later.

Let $\F$ denote the family of labeled copies of $T$ in $L$ that are induced subgraphs of $G$. 
Note that $|V(F_{r,t})|=rt+r+1$. By choosing $A$ to be sufficiently large,
we can ensure that $d\geq s(rt+r+1)^22^{(rt+r+1)+6}K^{rt+r}$ and hence by Lemma \ref{lem:countingtrees},
\begin{equation} \label{eq:F-lower}
|\F|\geq m(d/2)^{rt+r}.
\end{equation}

\begin{definition}
{\rm Given a $(t+1)$-star $D$ in $L$ with leaf set $S$, we say
that $D$ is {\it $C(H,s)$-heavy} if $|N^*_L(S)|\geq C(H,s)$; otherwise we say that
$D$ is {\it $C(H,s)$-light}.}
\end{definition}

\begin{definition}
{\rm We call a member of $\F$ {\it admissible} if it does not contain any $C(H,s)$-heavy $(t+1)$-star in $L$. }
\end{definition}

Let $\F'$ denote the family of admissible members of $\F$.  
Note that $H$ is 
a bipartite graph in which the neighborhood hypergraph induced by one part is  $(t+1)$-partite
$(t+1)$-uniform. Since $L$ does not contain a copy of $H$ that is induced in $G$, by Proposition \ref{prop:heavy-stars},
the number of $C(H,s)$-heavy $(t+1)$-stars in $L$ is at most $\varepsilon md^{t+1}$.
Since $\Delta(G)\leq Kd$, the number of members of $\F$ that contain a $C(H,s)$-heavy $(t+1)$-star in $L$ is at most $\varepsilon md^{t+1}(Kd)^{rt+r-(t+1)}
\leq \varepsilon K^{rt+r-t-1} m^{rt+r} < (1/2) m(\frac{d}{2})^{rt+r}$ for sufficiently small $\varepsilon$. Hence, we may assume
that 
\begin{equation} \label{eq:F-prime-bound}
|\F'|\geq \frac{1}{2}m\left(\frac{d}{2}\right)^{rt+r}\geq \frac{1}{2}\left(\frac{A}{4K}\right)^{rt+r} m^{rt},
\end{equation} 
where the last inequality uses \eqref{eq:L-assumptions}.

From now on, for any subfamily $\G$ of $\F'$ and for any vector $V$ of at most $rt$ vertices, we will use
$\G_V$ to denote the subfamily of members of $\G$ in which the image of the subvector of $\langle z_{1,1},\dots, z_{1,t},
\dots, z_{r,1},\dots, z_{r,t}\rangle$ consisting of the first $|V|$ entries equals $V$. 

\begin{lemma} \label{lem:robust-cleaning}
There exists a subfamily $\F''\subseteq \F'$ and index $i\in [r]$ such that 
\begin{enumerate}
\item For every vector $Z$ of $rt$ vertices in $L$ such that $\F''_Z$ is nonempty, all members of $\F''_Z$ map $y_i$ to the same vertex $v$.
\item $|\F''|\geq \frac{1}{16 \lambda r^2(r + 1)}
(\frac{A}{4K})^{rt+r} m^{rt+r}$.
\item  For every vector $Z$ of $rt$ vertices in $L$ such that $\F''_Z$ is nonempty, $|\F''_Z|\geq\frac{1}{32 \lambda r^2(r + 1)}
(\frac{A}{4K})^{rt+r}$.
\item For every proper
subtree $D$ contained in at least one member of $\F''$, the number of members of $\F''$ containing $D$ is at least $\beta d^{rt+r-e(D)}$, for $\beta = \frac{1}{K^{rt + r} 2^{3rt +3 r + 5} \lambda r^2(r + 1)}$. 
\end{enumerate}
\end{lemma}

\begin{proof}

Let $\lambda=\lambda(T,\ell)$ as given in Lemma \ref{lem:semi-induced-to-induced}. Suppose there is a vector $Z$ of $rt$ vertices for which $|\F'_Z| \leq  2\lambda(r+1)[C(H,s)]^r$. Remove from $\F'$ all the members of $\F'_Z$ from $\F'$. Continue this until all vectors $Z$ satisfy that either $\F'_Z = \emptyset$ or $|\F'_Z| \geq 2\lambda(r+1)[C(H,s)]^r$. For $A$ sufficiently large, no more than $\frac{1}{2}|\F'|$ members of $\F'$ have been removed. 
Consider any vector $Z$ of $rt$ vertices for which $\F'_Z\neq \emptyset$.  If we take a maximal collection $\D$ of members of $\F'_Z$
that are pairwise vertex disjoint outside $Z$, then $|\D|\leq \lambda$. Otherwise, $L$ contains a copy of $F^\lambda_Z$ that is semi-induced in $G$ and by Lemma \ref{lem:semi-induced-to-induced}, such a copy contains a copy of $F^\ell_R$ that is induced in $G$, contradicting our assumption. Hence, there is a set $S_Z$ of fewer than $\lambda(r+1)$ vertices outside $Z$ such that each member of $\F'_Z$ contains a vertex in $S_Z$. If some vertex $v$ in $L$ is the image of $x$ in more than $[C(H,s)]^r$ of the members of $\F'$, then these members contain a $C(H,s)$-heavy 
$(t+1)$-star  in $L$ with leaves $v$ and the images of $z_{i,1},\dots, z_{i,t}$ in $Z$ for some $i\in [r]$, contradicting these members being admissible. If $|\F'_Z|\geq 2\lambda(r+1)[C(H,s)]^r$, at least $\frac{1}{2}|\F'_Z|$ of the members of $\F'_Z$ uses a vertex in $S_Z$ as the image of one of $\{y_1,\dots, y_r\}$. 
By the pigeonhole principle, there exists some $i\in [r]$ such that at least $\frac{1}{2\lambda r(r+1)} |\F'_Z|$ members of $\F'_Z$ map $y_i$
to the same vertex $v$ in $S_Z$. Let $\cH_Z$ denote the subfamily of these members of $\F'_Z$,
and let  $c(Z)=i$. For each $i\in [r]$,
let $\Z_i$ denote the set of vectors $Z$ of $rt$ vertices with $c(Z)=i$.
By the pigeonhole principle, there exists some $i\in [r]$ such that $|\bigcup_{Z\in \Z_i} \F'_Z|\geq (1/r) |\F'|$. Let $\F''=\bigcup_{Z\in \Z_i} \cH_Z$.
Then 
\begin{equation} \label{eq:F-triple-prime}
|\F''|\geq \frac{1}{2\lambda r^2(r + 1)}|\F'|\geq \frac{1}{8\lambda r^2(r + 1)} \left(\frac{A}{4K}\right)^{rt+r} m^{rt}.
\end{equation}

Note that by definition, $\F''$ satisfies condition 1.

We iteratively clean $\F''$ using the following rules.
For convenience, we still denote the remaining subfamily of $\F''$ at 
each step by $\F''$.
\begin{enumerate}
\item If there is a vector $Z$ of $rt$ vertices such that $|\F''_Z|\leq \frac{1}{32 \lambda r^2(r + 1)}
(\frac{A}{4K})^{rt+r}$,
we remove all the members of $\F''_Z$ from $\F''$.
\item If there exists a proper subtree $D$ of some member of $\F''$ such that
the number of members of $\F''$ containing $D$ is fewer than $\beta d^{rt+r-e(D)}$
then remove all the members of $\F''$ containing $D$.
\end{enumerate}

Note that total number of members of $\F''$ removed by type 1 removals is at most $\frac{1}{32 \lambda r^2(r + 1)}
(\frac{A}{2K})^{rt+r} m^{rt}$.
Since $\Delta(L)\leq Kd$, the total number of members of members of $\F''$ removed by
type 2 removals is at most $2^{rt + r} \sum_{i=0}^{rt+r-1} m(Kd)^i\cdot \beta d^{rt+r-i}
\leq ((2K)^{rt+r}\beta) m d^{rt+r}$. By \eqref{eq:F-prime-bound}, by choosing $A$ to be sufficiently large, we can upper bound the total number of removals by $\frac{1}{32 \lambda r^2(r + 1)}
(\frac{A}{4K})^{rt+r} m^{rt}$. For the final $\F''$ we have $|\F''|\geq \frac{1}{16 \lambda r^2(r + 1)}
(\frac{A}{4K})^{rt+r} m^{rt}$.

\end{proof}

Now, let us fix such an $\F''$, and suppose without loss of generality the $i$ given by Lemma~\ref{lem:robust-cleaning} is $r$. 
By averaging, there exists a vector $V$ of $(r-1)t$ vertices such that 
\begin{equation} \label{eq:F-triple-prime-bound}
|\F''_U|\geq \frac{1}{16 \lambda r^2(r + 1)} \left(\frac{A}{4K}\right)^{rt+r} m^{t}.
\end{equation}

Let us fix such a $U$. Let $X$ denote the set of images of $x$ in the members of $\F''_U$ and $Y$ the set of images of $y_r$ in the members of $\F''_U$. Next, we prove a few properties about $\F''_U$.  If $U$, $V$ are vectors, let $U\vee V$ be the vector obtained by attaching $V$ to the end of $U$.
\begin{lemma} \label{lem:A-upper}
$|\F''_U|\leq |X|(Kd)^{t+1}\cdot (C(H,s))^r$.
\end{lemma}
\begin{proof}
Let $v\in X$. Among members of $\F''_U$ that has $x$ mapped to $v$ there are at most $(Kd)^{t+1}$ different vectors $V$ of $t$ vertices that
$\langle z_{r,1},\dots, z_{r,t}\rangle$ can be mapped to. For each such $V$, since members of $\F''_U$ are admissible, there are at most 
$(C(H, s))^r$ members of $\F''_U$ that map $\langle z_{1,1},\dots, z_{r,t}\}$ to $U\vee V$.
\end{proof}

By Lemma \ref{lem:A-upper} and \eqref{eq:F-triple-prime-bound}, we have
\begin{equation}\label{eq:xislarge}
|X|\geq \eta A^{rt + r}\frac{m^t}{d^{t+1}}, 
\end{equation}
 where  $\eta =\frac{1}{16 \lambda r^2 (r + 1) K^{t+1}[C(H,s)]^r}(\frac{1}{4K})^{rt+r}$ is independent of $A$.
 
\begin{lemma} \label{lem:F-triple-lower-bound}
We have $|\F''_U|\geq \beta |X|d^{t+1}$.
\end{lemma}
\begin{proof}
Let $S$ denote the subtree of $T$ induced by $V(T)\setminus \{y_r, z_{r,1},\dots, z_{r,t}\}$.
By definition of $X$, there are at least $|X|$ many different images of $S$
contained in members of $\F''_U$. By Lemma \ref{lem:robust-cleaning} condition 4, each such image extends to $\beta d^{t+1}$ members of $\F''$
all of which are in $\F''_U$ by the definition of $\F''_U$.
\end{proof}

We define a subfamily of $\F''_U$ as follows.
For each $V \in [V(G)]^t$, we call $V$ {\it good} if $|\F''_{U\vee V}| \geq \frac{|\F''_U|}{2m^t}$, and we call $V$ {\it bad} otherwise. Let $\F^*_U=\bigcup\{ \F''_{U\vee V}:
\mbox{ $V\in [V(G)]^t$ is good}\}$. By definition, $|\F^*_U| \geq \frac{1}{2}|\F''_U|$. 
Let $M$ be the subgraph of $L$ with $V(M)=X\cup Y$ whose edges are
the images of $xy_r$ in the members of $\F^*_U$. Since $L$ is bipartite, $M$ is bipartite with $(X,Y)$
being a bipartition.  Consider any edge $x'y$ in $M$, where $y\in Y$ is the image of $y_r$ in some member
$T\in \F''_{U\vee V}$ where $V\in [V(G)]^t$ is good. 
Since $T\in \F''_{U\vee V}$, by our
assumption, every member of $\F''_{U\vee V}$ 
maps $y_r$ to $y$. Consider the set $E$ of the images of $xy_r$ in these members. 
Since members of $\F''_{U\vee V}$ are admissible, $|E|\geq \frac{|\F''_{U\vee V}|}{[C(H,s)]^{r-1}}$. Hence, 
\begin{equation}\label{eq:M-min-degree}
d_M(y)\geq |E|\geq \frac{|\F''_{U\vee V}|}{[C(H,s)]^{r-1}}\geq \frac{|\F''_U|}{2m^t[C(H,s)]^{r-1}}\geq  \frac{\gamma |X|d^{t+1}}{m^t},
\end{equation}
where $\gamma=\frac{\beta}{2[C(H,s)]^{r-1}}$.
Note that $e(M)\geq |\F''_U|/ [C(H,s)^{r - 1}(Kd)^t])$ since each image of $xy_r$ can be contained in at most $C(H,s)^{r - 1}(Kd)^t$ many members of $\F_U''$.
Hence by Lemma \ref{lem:F-triple-lower-bound}, 
\[e(M)\geq \frac{\beta}{C(H,s)^{ r- 1}K^t} |X|d.\]

Applying Proposition \ref{prop:asymmetric} with $\delta_Y =\frac{\gamma |X|d^{t+1}}{m^t}$, $p=t+1$,
we have
\[e(M)\delta_Y^t\geq \alpha |X|^{t+1} \frac{d^{t^2+t+1}}{m^{t^2}},\]
where $\alpha=\frac{\beta\gamma^t}{C(H,s)K^t}$. Note that $H$ is an one side $(t+1)$-bounded bipartite graph.
Let $C_2, C_3$ be the constants returned by Proposition \ref{prop:asymmetric} for this $H$.
Since $d\geq \frac{A}{2K}m^{\frac{rt-1}{rt+r}}$, and $r\geq t+2$, one can check that $(t^2+t+1)\frac{rt-1}{rt+r}\geq t^2$.  By choosing $A$ to be sufficiently large, 
we can ensure $e(M)\delta_Y^t \geq C_3 |X|^{t+1}$. Furthermore, since $|X|\geq  \eta A^{rt + r}\frac{m^t}{d^{t+1}}$, we have that $\delta_Y \geq C_2$ by taking $A$ sufficiently large. 

Thus, by Proposition \ref{prop:asymmetric}, $L$ contains a copy of $H$ that is
induced in $G$, a contradiction.  This completes the proof of Theorem \ref{thm:height-two}.

\hfill $\Box$





\section{Proof of Theorem \ref{thm:asminus1}}

In this section, we let $T=T_{r, 1, 1}$ denote the tree formed by taking a spider of height $2$ with $r$ legs and then adding one leaf to the center. More specifically, $T$ has vertex set $\{a,b_1,\dots, b_r, c_1,\dots, c_r, b_{r+1}\}$ and edge set $\{ab_i:i\in [r]\}\cup\{b_ic_i: i\in [r]\}\cup \{ab_{r+1}\}$. Let $R$ be the set of leaves of $T$.
Let $\ell$ be given. Let $H=T_R^\ell$.
Let $A$ be a sufficiently large constant depending on $r,\ell,s$. Let $G$ be an $n$-vertex $K_{s,s}$-free graph. Let $X,Y$ be a partition
of $V(G)$ and assume that $e(G[X,Y])\geq A n^{2-\frac{r+1}{2r+1}}$. We show that if $A$ is taken to be sufficiently large then
$G[X,Y]$ must contain a copy of $H$ that is induced in $G$. 

Suppose for contradiction that $G[X,Y]$ does not contain a copy of $H$ that is
induced in $G$. Let $\alpha=1-\frac{r+1}{2r+1}=\frac{r}{2r+1}$ and let $K=2^{\frac{4}{\alpha}+2}$ as in Lemma \ref{lem:almost-reg},
By the lemma, $G[X,Y]$ contains an $m$-vertex $K$-almost-regular induced subgraph $L$ with $e(L)\geq \frac{C}{4}m^{1+\alpha}$, where $m=\Omega(n^{\frac{\alpha}{2\alpha+4}})$. Let $X_L=X\cap V(L)$ and $Y_L=Y\cap V(L)$. Let $G'=G[V(L)]$. Note that $G'[X_L,Y_L]=L$.
Let $d$ denote the minimum degree of $L$. By our assumption, 
\[d\geq \frac{A}{2K}m^\alpha=\frac{A}{2K}m^{\frac{r}{2r+1}} \mbox{ and }\Delta(L)\leq Kd.\]
Let $\ve = \frac{1}{3r}(4K)^{-3r}$.

Let $\lambda=\lambda(T,\ell)$ as given in Lemma \ref{lem:semi-induced-to-induced}.
Let $C = \lambda sr (rK \ve^{- 1})^{ r + s + 2} 2^{6 + 3s + 4r}$. 
 we will say a path $xyz$ in $L$ is {\it $C$-heavy} 
with respect to $L$ if the $|N_L^*(x, z)|\geq C$.
Let $\cH$ denote the family of paths of length two in $L$ that are induced in $G$
and are $C$-heavy with respect to $L$.
We will choose $A = \max\{ 2K (r^2 2^{4r +4}sK^{2r} \ve^{-r})^s, 8K(6r^2C^{r + 1}\lambda )^{\frac{1}{2r + 1}}\}$, and thus 
$d\geq (r^2 2^{4r +4}sK^{2r} \ve^{-r})^s$.

\begin{lemma}\label{lem:fewheavypaths}
We have $|\cH| \leq \ve m d^2$.
\end{lemma}
\begin{proof}
Suppose for contradiction that $|\cH|>\ve md^2$.
By averaging, there is a vertex $v$ which is the middle vertex of at least $\ve d^2$ many 
members of $\cH$. We form an auxiliary graph $\D$ on $N_L(v)$ by joining $xy$ if $xvy\in \cH$. 
Then $e(\D)=|\cH|\geq \ve d^2$ and $|V(\D)|\leq Kd$. Since $|V(\D)|\leq Kd$, by iteratively removing a vertex of degree less than $\frac{\ve}{2K}d$,
we can find a subgraph $\D'$ with minimum degree at least $\frac{\ve}{2K}d$ such that
$e(\D')\geq \frac{1}{2} \ve d^2$.

Let 
\[B_0=\{w\in V(L): |N_G(w)\cap V(\D')| \geq \frac{\ve}{8K}d \}.\]
Since $|V(\D')|\geq \frac{\ve}{2K}d$, $d\geq s(16K^2 \ve^{-1})^{s+1}$, and $|V(\D')| \leq Kd$,  
by Lemma~\ref{lem:small-bad-set}, $|B_0|\leq 16sK^2 \ve^{-1}$. 

Given vertices $x,y\in V(\D')$, let $B(x,y)=\{w\in V(L): |N_G(w)\cap N^*_L(x,y)|\geq \frac{1}{4r} |N^*_L(x,y)|$. Since $G$ is $K_{s,s}$-free and $|N^*_L(x,y)|\geq C\geq s(8r)^s$, 
by Lemma \ref{lem:small-bad-set}, $|B(x,y)|\leq 8sr$ for any $x,y\in V(\D')$.
Given a vertex $x\in V(\D')$ and a vertex $w\in V(L)$, we say that $w$ is {\it infeasible} for  $x$  if $w\in B(x,y)$ for at least $\frac{1}{2}|N_{\D'}(x)|$ many $y \in N_{\D'}(x)$.
Letting $I(x)$ be the set of infeasible vertices for $x$, we have 
\[|I(x)|\cdot \frac{1}{2}|N_{\D'}(x)|\leq \sum_{y\in N_{\D'}(x)} |B(x,y)| \leq 8sr |N_{\D'}(x)|.\]
Hence $|I(x)|\leq 16sr$ for every $x\in V(\D')$.
Let $\Pc$ be a subfamily of members of $\cH$ obtained as follows.
For all $xz\in E(\D')$, we include a member $xwz\in \cH$ in $\Pc$
if $w\notin B_0\cup I(x)\cup I(z)$.
Since $e(\D')\geq \frac{\ve}{2} d^2$ and $C \geq 64 srK^2 \ve^{-1}$,
by our discussion above,  $|\Pc| \geq \frac{\ve}{2}d^2 C/2$.

By averaging, there exists a $xw$ that is in at least $\frac{\ve C}{4K^2}$ members of $\Pc$
of the form $xwz$. Fix such $xw$ and let $Z$ denote the set of $z$ in these $xwz$.  
Let $A(z) = \{v \in N_{\D'}(z) \setminus N_G(w) : w \not \in B(z, v)\}$. 
By our minimum degree condition,  $w\notin B_0$, and $w$ is not infeasible for $z$, for each $z\in Z$, we have $|A(z)| \geq \frac{\ve}{8K}d$. 
For each $z\in Z$, let $\T_z$ denote the collection of labeled $r$-stars in $\D'$ with center $z$ and leaves in $A(z)$. 
We note that $|\T_z|\geq \prod_{i = 0}^{r - 1}(\frac{\ve}{8K} d - i) \geq (\frac{\ve}{16K})^r$.
Observe that by construction, if $S\in \T_z$, there is no edge in $G$ between $\{w,z\}$ and $V(S)\setminus \{z\}$
and the only potential edges are inside $V(S) \setminus \{z\}$. 

Since $G$ is $K_{s,s}$-free, by Lemma \ref{lem:KST},
$G[A(z)]$ has at most $2s(K d)^{2 - 1/s}$ edges. Hence, the number of members of $\T_z$ that contains two leaves
that are adjacent in $G$ is at most $r^2sK^{r}d^{r - 1/s}$. Since $d \geq (r^2 2^{4r +4}sK^r \ve^{-r})^s$, we have at least $\frac{1}{2}  (\frac{\ve}{16K}d)^r$ members of $\T_z$ whose leaves induce no edge in $G$. Let $\Sc_z$ denote the collection of these members.
We have $\sum_{z\in Z} |\Sc_z|\geq |Z| \frac{1}{2}  (\frac{\ve}{16K}d)^r$.
Since the number of $r$-tuples in $\D'$ is at most $(Kd)^r$, there is some $r$-tuple $\langle x_1,\dots, x_r\rangle$ that is the leaf vector of  at least $\frac{\ve C}{4K^2}\frac{1}{2}(\frac{\ve}{16K})^r \geq \lambda$ members of $\bigcup_{z\in Z} \Sc_z$.

Let $K_1, \dots K_\lambda$ be $\lambda$ of these members. Recall that $T$ has vertex set $\{a,b_1,\dots, b_r, c_1,\dots, c_r, b_{r+1}\}$ and edge set $\{ab_i:i\in [r]\}\cup\{b_ic_i: i\in [r]\}\cup \{ab_{r+1}\}$. Using $K_1,\dots, K_\lambda$ we will build copies $T_1,\dots, T_\lambda$ of $T$ in $L$ that are induced subgraphs of $G$ which all map $c_1,\dots, c_r, b_{r+1}$ to $x_1,\dots, x_r, w$, respectively, but are
otherwise vertex disjoint from each other. But then by Lemma \ref{lem:semi-induced-to-induced}, $L$ would contain a copy of $H$ that is induced in $G$, contradicting our assumption. Thus, it suffices to show that we can find these $T_i$'s.
Let $i\in [\lambda]$. Suppose we have found $T_1,\dots, T_{i-1}$, we explain how to find $T_i$.
Let $z$ be the center of $K_i$. In $T_i$, $z$ will play the role of $a$. For each $\ell\in [r]$,  recall that $B(z, x_\ell)$ is the set of vertices $y$ in $L$ such that $|N_G(y) \cap N^*_L(z,x_\ell))| \geq \frac{1}{4r} |N^*_L(z, x_\ell) |$. We have shown earlier that $|B(z, x_\ell)| \leq 8sr$. We pick $y_1,\dots, y_r$ iteratively as follows, so that each $y_\ell$ will play the role of $b_\ell$. 
For each $\ell \in [r]$,  supposing we have picked $y_1,\dots, y_{\ell-1}$, 
let $y_\ell$ be any element from $$[N^*_L(z, x_\ell)\setminus N_G(w)] \setminus  \left(\bigcup_{j = 1}^{\ell - 1} N_G(y_j) \cup  \bigcup_{j = 1}^r B(z, x_j) \cup \bigcup_{j = 1}^{i - 1} V(T_j) \cup \{y_1, \dots y_{\ell - 1}\}\right).$$ Note that $|N_G(w) \cap N^*_L(z, x_\ell)|\leq \frac{1}{4r}|N^*_L(z, x_\ell)|$ by our choice of $z, x_\ell$ satisfying $w \not \in B(z, x_\ell)$. Since $y_1, \dots y_{\ell - 1} \not \in B(z, x_\ell)$, we have that $|(\bigcup_{j = 1}^{i - 1}N_G(y_i)) \cap N^*_L(z, x_\ell)| \leq \frac{1}{4} |N^*_L(z, x_\ell)|$. Since $|B(z, x_j)|\leq 8sr$ for each $j$ and and we have at most selected $\lambda(8r)$ vertices so far in this process, as long as $C \geq 64sr\lambda$, we have enough choices for each $y_\ell$ in succession.
This completes the proof of Lemma \ref{lem:fewheavypaths}.
\end{proof}

Now, we are ready to complete the proof of Theorem \ref{thm:asminus1}.
Let $\F$ denote the family of labeled copies of $T$ in $L$ that are induced subgraphs of $G$. Since we chose $A$ to make $d$ sufficiently large,
by Lemma~\ref{lem:countingtrees}, 
\[|\F|\geq m\left(\frac{d}{2}\right)^{2r + 1}.\]
By Lemma \ref{lem:fewheavypaths}, $|\cH|\leq \ve m d^2$. So the number of members of $\F$ that contains a member of $\cH$
is at most $3r \ve m d^2 (Kd)^{2r - 1}$. Let $\F'$ denote the subfamily of members of $\F$ that do not contain any member of $\cH$.
Then, as $d \geq \frac{A}{2K} m^{\alpha}$ and $A$ is sufficiently large,
\[|\F'|\geq \frac{1}{2} m \left(\frac{d}{2}\right)^{2r + 1} \geq  \frac{1}{2} \left(\frac{A}{4K} \right)^{2r + 1} m^{r + 1} \]
By the pigeonhole principle, we have find a vector $Z$ of $r+1$ vertices in $L$ that is the leaf vector of at least $(\frac{|A|}{8K})^{2r + 1}$ 
members of $\F'$, call the collection of these members $\F_Z'$.
 If we take a maximal collection $\D$ of members of $\F_Z'$
that are pairwise vertex disjoint outside $Z$, then $|\D|< \lambda$. Otherwise, $L$ contains a copy of $T^\lambda_R$ that is semi-induced in $G$ and by Lemma \ref{lem:semi-induced-to-induced}, such a copy contains a copy of $T^\ell_R$ that is induced in $G$, contradicting our assumption about $L$. Hence, there is a set $S_Z$ of fewer than $\lambda(r+1)$ vertices outside $Z$ such that each member of $\F_Z'$ contains a vertex in $S_Z$.  

Consider any vertex $v$ in $S_Z$.
Since no member of $\F_Z'$ contains a path of length two that is $C$-heavy with respect to $L$, a vertex in $S_Z$
can play the role of $a$ in at most $C^r$ members of $\F_Z'$. Also, for each $i\in [r]$. $v$ can play the role of $b_i$
in at most $C^r$ members of $\F_Z'$. Hence, we must have $|\F_Z'|\leq \lambda(r+1)^2 C^r<(A/8K)^{2r+1}$, contradicting our earlier claim about $\F_Z'$. This completes the proof of Theorem \ref{thm:asminus1}. \hfill $\Box$



\end{document}